\documentclass{amsart}

\usepackage{amsmath}
\usepackage{amsthm}
\usepackage{amssymb}
\usepackage{amsfonts}
\usepackage{longtable}

\theoremstyle{plain}

\newtheorem{theorem}{Theorem}[section]
\newtheorem{corollary}[theorem]{Corollary}
\newtheorem{lemma}[theorem]{Lemma}

\newtheorem{proposition}[theorem]{Proposition}

{}
{}
{}
\theoremstyle{definition}
\newtheorem{definition}[theorem]{Definition}
\newtheorem*{notation}{Notation}
\theoremstyle{remark}
\newtheorem{remark}[theorem]{Remark}

\newlength{\myVSpace}
\begin{document}

\title{SIGNATURE PAIRS FOR GROUP-INVARIANT HERMITIAN POLYNOMIALS}

\author{DUSTY GRUNDMEIER}

\address{Department of Mathematics, University of Illinois\\ 1409 W. Green St., Urbana,
IL 61801, USA\\
grundmer@uiuc.edu}

\maketitle

\begin{abstract}
We study the signature pair for certain group-invariant Hermitian polynomials arising in CR geometry.  In particular, we determine the signature pair for the finite subgroups of $SU(2)$.  We introduce the asymptotic positivity ratio and compute it for cyclic subgroups of $U(2)$.  We calculate the signature pair for dihedral subgroups of $U(2)$.
\end{abstract}



\section{Introduction}
The purpose of this paper is to determine the signature pair (defined momentarily) for Hermitian polynomials arising from group-invariant CR mappings from spheres to hyperquadrics.  Let $\Gamma$ be a finite subgroup of the unitary group $U(n)$.  Let $S^{2n-1}$ denote the unit sphere in $\mathbb{C}^{n}$.  We assume $n\geq2$.  

A natural question is: when does there exist a non-constant $\Gamma$-invariant CR mapping from $S^{2n-1}$ to $S^{2N-1}$?  Forstneri{\v c} showed that a smooth CR mapping from $S^{2n-1}$ to $S^{2N-1}$ must be a rational mapping \cite{F2}.  He also found restrictions on the possible groups $\Gamma$ for which such a rational map exists \cite{F1}.  Lichtblau \cite{L} proved that for non-constant  $\Gamma$-invariant rational maps between spheres to exist, $\Gamma$ must be cyclic.  Later D'Angelo and Lichtblau \cite{D0,DL} answered this question by finding the complete list of cyclic $\Gamma$ for which such a rational map exists.  To do so they introduced the $\Gamma$-invariant Hermitian polynomial defined by \begin{equation}\label{e:phigamma}\Phi_{\Gamma}(z, \bar{z})=1-\prod_{\gamma \in \Gamma}{\left(1-\langle \gamma z, z \rangle \right)}.\end{equation}  This polynomial also determines, by diagonalizing its underlying Hermitian matrix of coefficients, a group-invariant CR map from a sphere to a hyperquadric \cite{D2}.  Let $N(\Gamma)$, $N^{+}(\Gamma)$, and $N^{-}(\Gamma)$ be the numbers of total eigenvalues, positive eigenvalues, and negative eigenvalues respectively of this underlying Hermitian matrix of coefficients of $\Phi_{\Gamma}$.  We refer the reader to section 2 for precise definitions.  The {\it signature pair} $S(\Gamma)$ is $$S(\Gamma)=(N^{+}(\Gamma), N^{-}(\Gamma)),$$ and the {\it positivity ratio} is $$L(\Gamma)=\frac{N^{+}(\Gamma)}{N(\Gamma)}.$$  Because the positivity ratio is often difficult to compute, we study its asymptotic behavior.  For a family of subgroups $\Gamma_p$ of $U(n)$, we define the {\it asymptotic positivity ratio} to be $$\lim_{p\to \infty}{L(\Gamma_p)}.$$  Here the index $p$ is closely related to the order of the group (see section 3).  The polynomial $\Phi_\Gamma$ canonically induces a CR mapping to a hyperquadric with $N^{+}(\Gamma)$ positive eigenvalues and $N^{-}(\Gamma)$ negative eigenvalues in its defining equation (see \cite{D2}).  We do not pursue this aspect of the polynomial $\Phi_{\Gamma}$.

The main results of this paper compute the signature pair for finite subgroups of $SU(2)$, calculate the asymptotic positivity ratio for cyclic subgroups of $U(2)$, and determine the signature pair for the dihedral groups in $U(2)$.  In this paper we work in $U(2)$; however, many of the results can be extended to $U(n)$.  The results for arbitrary $n$ will appear in the author's doctoral thesis (\cite{G}).  We first restrict to finite group subgroups $SU(2)$. The only finite subgroups of $SU(2)$ are isomorphic to one of the following:
\begin{itemize}
\item  Cyclic group of order $p$: $C_p=\left< a \; | \; a^p=1\right>$.
\item  Binary Dihedral group of order $4p$: $$Q_p=\left< a, b \; | a^p=b^2, a^{2p}=1, b^{-1}ab=a^{-1}\right>.$$
\item  Binary Tetrahedral group of order 24: $T=\left< a, b \;   | \; a^3=b^3=(a b)^2\right>$.
\item  Binary Octahedral group of order 48: $O=\left< a, b \;   | \; a^4=b^3=(a b)^2\right>$.
\item  Binary Icosahedral group of order 120: $I=\left< a, b \;   | \; a^5=b^3=(a b)^2\right>$.
\end{itemize}
The first main result of this paper computes the signature pair for finite subgroups of $SU(2)$.  While we are primarily interested in families of groups, for completeness we consider the signature pair for the three exceptional groups.  Using Mathematica \cite{Mathematica} we obtain the signature pair for the binary polyhedral groups.  Here is the complete story for the subgroups of $SU(2)$.

\begin{theorem}  \label{mt:subgroups}  Let $\Gamma$ be a finite subgroup of $SU(2)$.
\begin{enumerate}
\item[1.] If $\Gamma$ is isomorphic to a cyclic group of order $p$, then $$S(\Gamma)=\left( \left \lfloor \frac{p+2}{4} \right \rfloor+2, \left \lfloor \frac{p}{4} \right \rfloor \right).$$
\item[2.] If $\Gamma$ is isomorphic to a binary dihedral group of order $4p$, then $$S(\Gamma)=\left( \left \lfloor \frac{p}{2} \right \rfloor+p+2, \left \lfloor \frac{p-1}{2} \right \rfloor +1\right).$$
\item[3.] If $\Gamma$ is isomorphic to a binary tetrahedral group of order 24, then $$S(\Gamma)=\left( 9,5 \right).$$
\item[4.] If $\Gamma$ is isomorphic to a binary octahedral group of order 48, then $$S(\Gamma)=\left( 17,9 \right).$$
\item[5.] If $\Gamma$ is isomorphic to a binary icosahedral group of order 120, then $$S(\Gamma)=\left( 40,22 \right).$$
\end{enumerate}
\end{theorem}

We next consider cyclic and dihedral groups in $U(2)$.  For a cyclic subgroup of order $p$ in $U(2)$, several different signature pairs are possible.  For dihedral subgroups of $U(2)$, the signature pair depends only on the isomorphism type of the group.  For the cyclic group $C_p$ with $p$ elements, we consider the group representations $\pi : C_p \to \Gamma(p,q) <U(2)$ generated by \begin{equation}\label{e:gammagen}s\mapsto \begin{pmatrix} 
\omega & 0\\ 
0 & \omega^q
\end{pmatrix}\end{equation} where $\omega$ is a primitive $p$-th root of unity and $s$ is an element of order $p$ in $C_p$.  Up to conjugation, every finite cyclic subgroup of $U(2)$ is of the form $\Gamma(p,q)$ for some $p$ and $q$.  We compute the asymptotic positivity ratio of $\Gamma(p,q)$ for any $q$; we show that the asymptotic positivity ratio is a rational expression depending on $q$.  Further we take the limit as $q$ goes to infinity to obtain the following theorem.

\begin{theorem}  \label{mt:cyclic}  Let $\Gamma(p,q)$ be as in \eqref{e:gammagen}, then 
$$\lim_{p\to \infty}{L(\Gamma(p,q))} = \begin{cases}
\frac{3q+1}{4q} & \text{if $q$ is odd},\\
\frac{3q-2}{4(q-1)} & \text{if $q$ is even},
\end{cases}$$ and hence 
$$\lim_{q\to \infty}{\lim_{p\to \infty}{L(\Gamma(p,q))}}= \frac{3}{4}.$$ 
\end{theorem}

The details of the proof appear in section 4, but we give a short description now.  First we recall from \cite{D4} the weight of a monomial appearing in $\Phi_{\Gamma(p,q)}$.  Then we find bounds for the total number of terms and the number of terms of each weight.  Using this information, we calculate bounds on the fraction of terms of odd weight and the fraction of terms of even weight.  Then we show that asymptotically the numbers of even and odd weight terms are equal.  We then interpret a result of Loehr, Warrington, and Wilf \cite{LWW} in terms of weights.  Their result implies that all the odd weight terms are positive, and the even weight terms alternate sign.  Since $\Gamma(p,q)$ is diagonal, the number of terms is the same as the number of eigenvalues.  It follows that the limit as $q$ goes to infinity of the asymptotic positivity ratio is $\frac{3}{4}$. 

The third main result calculates the signature pair for dihedral subgroups of $U(2)$.  
\begin{theorem}  \label{mt:dihedral} Let $\Delta_p$ be a dihedral subgroup of order $2p$ in $U(2)$, then
$$S(\Delta_p)=\left( \left \lfloor \frac{p}{2} \right \rfloor +\left \lfloor \frac{p}{4} \right \rfloor +2, \left \lfloor \frac{3(p+1)}{4} \right \rfloor \right),$$
and hence
$$\lim_{p\to \infty}{L(\Delta_p)}=\frac{1}{2}.$$
\end{theorem}  

We conclude the introduction by outlining the rest of the paper.  In section 2 we give relevant definitions, introduce the weight of a polynomial, and prove some basic facts about Hermitian polynomials.  In section 3 we compute the signature pairs for finite subgroups of $SU(2)$.  In sections 4 and 5 we prove the main results for subgroups of $U(2)$.  In section 6 we recall some basic definitions from representation theory, and we show that in this context the polynomials $\Phi_{\Gamma(p,q)}$ are an alternating sum of orbit Chern classes. 

\section{Definitions and Preliminaries}
In this section we recall some basic facts about unitary representations and Hermitian polynomials.  We begin by defining Hermitian polynomials.  
\begin{definition}  Let $R:\mathbb{C}^n \times \mathbb{C}^n \to \mathbb{C}$ be a polynomial.  We call $R$ Hermitian if $$R(z,\bar{w})=\overline{R(w,\bar{z})}.$$
\end{definition}
Given any polynomial $$r(z, \bar{w})=\sum c_{\alpha \beta} z^{\alpha} \bar{w}^{\beta},$$ then $r$ is Hermitian if and only if the matrix $(c_{\alpha \beta})$ is Hermitian if and only if $r(z, \bar{z})$ is real-valued (see \cite{D0}).  We call $(c_{\alpha \beta})$ the {\it underlying matrix of $r$}.  We define $N(r)$, $N^{+}(r)$, and $N^{-}(r)$ to be the numbers of total eigenvalues, positive eigenvalues, and negative eigenvalues respectively of $(c_{\alpha \beta})$.  We define the {\it signature pair} $S(r)$ of a Hermitian polynomial to be the pair $S(r)=(N^{+}(r), N^{-}(r))$.  Define the {\it positivity ratio} by $L(r)=\frac{N^{+}(r)}{N(r)}$.  We recall the definition of the polynomial $\Phi_{\Gamma}$:
\begin{equation*}\Phi_{\Gamma}(z, \bar{z})=1-\prod_{\gamma \in \Gamma}{\left(1-\langle \gamma z, z \rangle \right)}.\end{equation*}
\begin{notation}    For any $\Gamma<U(n)$, put $N^{+}(\Gamma)=N^{+}(\Phi_\Gamma)$, $N^{-}(\Gamma)=N^{-}(\Phi_\Gamma)$, and $N(\Gamma)=N^{+}(\Gamma)+N^{-}(\Gamma)$.  Put $S(\Gamma)=(N^{+}(\Gamma), N^{-}(\Gamma))$, and $L(\Gamma)=\frac{N^{+}(\Gamma)}{N(\Gamma)}$.  
\end{notation}

For families of subgroups $\Gamma_p$ of $U(n)$, we define the {\it asymptotic positivity ratio} to be $\lim_{p\to \infty}{L(\Gamma_p)}$.

\begin{definition}  Let $C_p$ be a cyclic group of order $p$ with generator $s$.  Define a unitary representation $\pi : C_p \to U(2)$ by $$\pi(s) = 
\begin{pmatrix} 
\omega & 0\\ 
0 & \omega^{q}
\end{pmatrix}$$
where $\omega$ is a $p$-th primitive root of unity.  Let $\Gamma(p,q)=\pi(C_p)$.
\end{definition}

\begin{definition}  Two group representations $\pi_1 : G \to U(n)$ and $\pi_2 : G \to U(n)$ are called {\it equivalent} if there exists an element $A \in U(n)$ such that $$A \pi_1(g) A^{-1}= \pi_2(g)$$ for every $g \in G$.
\end{definition}
  
\begin{definition}A polynomial $f(x,y)$ has {\it weight} $j$ with respect to $\Gamma(p,q)$ if $$f(\lambda x, \lambda^q y) = \lambda^{jp} f(x,y)$$ for all $\lambda \in \mathbb{C}$.  In particular, the monomial $x^ay^b$ has weight $j$ if $a+qb=jp$.
\end{definition}

Because $\Phi_{\Gamma(p,q)}$ depends only on $|z_1|^2$ and $|z_2|^2$, we define the polynomial $f_{p,q}$ by 
\begin{equation}\label{e:fpq} f_{p,q}(\left| z_1 \right|^2, \left| z_2 \right|^2)=f_{p,q}(x,y)=1-\prod_{j=0}^{p-1}{\left( 1-\omega^{j}x-\omega^{qj}y\right)}=\Phi_{\Gamma(p,q)}(z,\bar{z}).
\end{equation}  

The polynomials $f_{p,q}$ have many interesting number-theoretic and combinatorial properties (see \cite{D4,D3,LWW,M,O}). In the case $q=1$, we obtain $f_{p,1}=(x+y)^p$.  In the case $q=2$, we get a variant of the Chebyshev polynomials.  The importance of the polynomials in these two cases motivates the study of $f_{p,q}$ for higher $q$.  For the reader's convenience, we list $f_{p,4}$ for $1\leq p \leq 9$ in Table 1.

\begin{table}[ht]
\caption{List of $f_{p,4}$ for $1\leq p\leq 9$.}
\hrule
{\begin{tabular}{cl}
$f_{1,4}(x,y)$   & = $x+y$ \\
$f_{2,4}(x,y)$   & = $x^2+2y-y^2$ \\
$f_{3,4}(x,y)$   & = $x^3+3x^2y+3xy^2+y^3$ \\
$f_{4,4}(x,y)$   & = $x^4+4y-6y^2+4y^3-y^4$ \\ 
$f_{5,4}(x,y)$   & = $x^5+5xy-5x^2y^2+y^5$ \\
$f_{6,4}(x,y)$   & = $x^6+6x^2y-3x^4y^2+2y^3+3x^2y^4-y^6$\\
$f_{7,4}(x,y)$   & = $x^7+7x^3y+14x^2y^3+7xy^5+y^7$ \\
$f_{8,4}(x,y)$   & = $x^8+8x^4y+4y^2+8x^4y^3-6y^4+4y^6-y^8$ \\
$f_{9,4}(x,y)$   & = $x^9+9x^5y+9xy^2+3x^6y^3-18x^2y^4+3x^3y^6+y^9$\\ 
\end{tabular}}
\hrule
\end{table}

We now summarize some properties of the $f_{p,q}$.  A beautiful result from \cite{D4} is that for all $q$, $f_{p,q}$ is congruent to $(x+y)^p$ mod $(p)$ if and only if $p$ is prime.  We naturally ask what other properties of $(x+y)^p$ generalize to $f_{p,q}$ for all $q$.  In \cite{D1} D'Angelo constructs the $f_{p,q}$ and shows that the coefficients are integers.  The $f_{p,2}$ polynomials have an extremal property studied in \cite{D3}.  Dilcher and Stolarsky consider a generalization of the $f_{p,2}$ polynomials in \cite{DS}.  Osler uses a variant of the $f_{p,2}$ to denest radicals \cite{O}.  Musiker uses the $f_{p,2}$ polynomials while studying the combinatorics of elliptic curves \cite{M}.  Loehr, Warrington, and Wilf give a combinatorial interpretation for the coefficients of $f_{p,q}$ using circulant determinants \cite{LWW}.  They also gave a simple method of determining the sign of each term in $f_{p,q}$, which we use to calculate the asymptotic positivity ratio for the $\Gamma(p,q)$.  They also study the asymptotics of the largest coefficient appearing in $f_{p,q}$.  In \cite{D5} D'Angelo uses methods from complex analysis to obtain asymptotic information; for example, he gives an asymptotic formula for the sum of the coefficients of $f_{p,q}$.  

The following basic property of Hermitian polynomials will be needed in the next section.  The signature pair for Hermitian polynomials is unchanged under a change of basis, and therefore equivalent representations have the same signature pair.

\begin{proposition}  \label{p:changebasis}  Given a Hermitian polynomial $r(z, \bar{z})$, then for every $U\in U(n)$ we have $S(r)=S(r \circ U)$.
\end{proposition}

\begin{proof}  Let $r(z, \overline{z})$ be a Hermitian polynomial of degree $d$; using multi-index notation, we write $$r(z, \overline{z})=\sum_{|\alpha|, |\beta|\leq d} c_{\alpha \beta} z^{\alpha} \overline{z}^{\beta}.$$  The polynomial $r(z, \overline{z})$ is a Hermitian form on the vector space of polynomials of degree at most $d$. Composing with $U$, we have $$r(Uz, \overline{Uz})=\sum c_{\alpha \beta} (Uz)^{\alpha} (\overline{Uz})^{\beta}.$$  Since $U$ is non-singular and the monomials $z^\alpha$ form a basis of the vector space of polynomials of degree less than $d$ in $z$, then $(Uz)^{\alpha}$ also form a basis of the vector space of polynomials of degree at most $d$.  Thus by Sylvester's Law (see page 223 of \cite{Horn}), $r(z, \overline{z})$ and $r(Uz, \overline{Uz})$ have the same numbers of eigenvalues of each sign.
\end{proof}

\begin{corollary}  \label{c:invsig}  If $\pi_1 : G \to U(n)$ and $\pi_2 : G \to U(n)$ are equivalent representations, then $S(\pi_1(G))= S(\pi_2(G))$.
\end{corollary}
\begin{proof}
The result follows by a change of coordinates.  Let $g\in G$, then there exists $A\in U(n)$ such that $$A \pi_1(g) A^{-1}= \pi_2(g).$$  Thus
\begin{eqnarray*}
\Phi_{\pi_1(G)}(z, \bar{z}) & = & 1 - \prod_{\gamma \in \pi_1(G)}{\left(1-\langle \gamma z, z \rangle \right)}\\
& = & 1 - \prod_{g \in G}{\left(1-\langle \pi_1(g) z, z \rangle \right)}\\
& = & 1 - \prod_{g \in G}{\left(1-\langle A^{-1} \pi_2(g) A z, z \rangle \right)}\\
& = & 1 - \prod_{g \in G}{\left(1-\langle \pi_2(g) A z, A z \rangle \right)}\\
& = & 1 - \prod_{g \in G}{\left(1-\langle \pi_2(g) w, w \rangle \right)}\\
& = & \Phi_{\pi_2(G)}(w, \bar{w})\\
\end{eqnarray*}
where $w=Az$. By Proposition \ref{p:changebasis}, the signature pair of a Hermitian polynomial is invariant under a change of coordinates, and the result follows.
\end{proof}

\section{Subgroups of $SU(2)$}
Given a unitary representation $\pi : G \to U(n)$, a natural question is: for fixed $n$, what are the possible finite groups $G$?  For $n=1$, the only finite groups are cyclic.  For $n>1$, the question becomes difficult.  In this paper we restrict to the case $n=2$.  When $n=2$, Du Val \cite{Du1} classified the finite groups while studying what are now called Du Val singularities.  He found nine families of groups.  We further restrict to $SU(2)$, and in this case, five types of subgroups arise:
\begin{itemize}
\item  Cyclic group of order $p$: $C_p:=\left< a \; | \; a^p=1\right>$.
\item  Binary Dihedral group of order $4p$: $Q_p:=\left< a, b \; | \; a^p=b^2, a^{2p}=1, b^{-1}ab=a^{-1}\right>$.
\item  Binary Tetrahedral group of order 24: $T:=\left< a, b \;   | \; a^3=b^3=(a b)^2\right>$.
\item  Binary Octahedral group of order 48: $O:=\left< a, b \;   | \; a^4=b^3=(a b)^2\right>$.
\item  Binary Icosahedral group of order 120: $I:=\left< a, b \;   | \; a^5=b^3=(a b)^2\right>$.
\end{itemize}

The purpose of this section is to give a complete analysis of the signature pair for each of these groups.

\subsection{Cyclic Groups}

Let $\Gamma$ be a cyclic subgroup of order $p$ in $SU(2)$.  Let $A$ be a generator of $\Gamma$ in $SU(2)$.  By the results of section 2, we can diagonalize $A$ without affecting the signature pair.  Thus it suffices to consider $A$ of the form $$\begin{pmatrix}
\omega^{a} & 0\\ 
0 & \omega^{b}
\end{pmatrix}$$ where $\omega$ is a $p$-th root of unity.  Since $A$ is in $SU(2)$, we also know that $a+b=p$.  Moreover, for $A$ to have order $p$, then $a$, $b$, and $p$ must be relatively prime.  Then $a$, $p-a$, and $p$ are relatively prime and hence $\omega^{a j}=\omega$ for some $j$.  Thus we can always choose $a$ to be 1 without loss of generality.  Hence $b=p-1$, and then $A$ generates $\Gamma(p,p-1)$.  Therefore up to conjugation, $\Gamma(p,p-1)$ is the only cyclic subgroup of order $p$ in $SU(2)$.    Thus the only possible signature pair for a cyclic subgroup of order $p$ in $SU(2)$ is given by $\Gamma(p,p-1)$.  We recall some useful facts about $\Phi_{\Gamma(p,p-1)}$ from \cite{D2}.  We will use these properties for computing the asymptotic positivity ratio for other groups.  

\begin{theorem} {({D'Angelo,} \cite{D2})}  \label{thm:cyclicproposition}  The following hold for $\Phi_{\Gamma(p,p-1)}$.
\begin{enumerate}
\item We have the following exact formula:
\begin{align*}\Phi_{\Gamma(p,p-1)}=1+|z_1|^{2p} +|z_2|^{2p}&-\left( \frac{1+\sqrt{1-4 |z_1|^2 |z_2|^2}}{2} \right)^p\\ &- \left( \frac{1-\sqrt{1-4|z_1|^2 |z_2|^2}}{2}\right)^p.\end{align*}
\item  The coefficients $c_{p,j}$ in the following formula are positive integers: $$\Phi_{\Gamma(p,p-1)}(z,\bar{z})= |z_1|^{2p}+|z_2|^{2p}+\sum_{j=1}^{\left \lfloor \frac{p}{2} \right \rfloor}{(-1)^{j-1}c_{p,j} |z_1|^{2j}|z_2|^{2j}}.$$
\item  These coefficients are given by $$c_{p,j}= \frac{p}{p-j}\binom{p-j}{j}.$$
\item Finally the signature pair is $$S(\Gamma(p,p-1))=\left(\left \lfloor \frac{p+2}{4} \right \rfloor + 2,\left \lfloor \frac{p}{4} \right \rfloor \right).$$
\end{enumerate}
\end{theorem}

\begin{remark}  D'Angelo also showed that the coefficients of $f_{p, p-1}$ are the same as the coefficients of $f_{p,2}$ up to sign.  More generally the coefficients of $f_{p, q}$ are the same as the coefficients of $f_{p,p-q+1}$ up to sign.
\end{remark}

\begin{corollary} The asymptotic positivity for cyclic subgroups of $SU(2)$ is $$\lim_{p\to \infty}{L(\Gamma(p,p-1))}=\frac{1}{2}.$$
\end{corollary}

\subsection{Binary Dihedral Groups}

Consider the binary dihedral groups $$Q_p:=\left< a, b \; | a^p=b^2, a^{2p}=1, b^{-1}ab=a^{-1}\right>.$$  The group $Q_p$ has order $4p$.  Let $\eta : Q_p \to SU(2)$ be the faithful representation generated by 
\begin{eqnarray*}
\eta(a) & = & 
\begin{pmatrix} 
\omega & 0\\ 
0 & \omega^{-1}
\end{pmatrix}\\
\eta(b) & = & 
\begin{pmatrix} 
0 & 1\\ 
-1 & 0
\end{pmatrix}\\
\end{eqnarray*}
where $\omega$ is a $2p$-th primitive root of unity.  Here set $\Lambda_p=\eta(Q_p)$.

Before stating the main results of this section we illustrate the techniques with an example.  We compute the number of positive and negative eigenvalues of $\Phi_{\Lambda_2}(z, \bar{z})$.  Expanding the product in the definition we get 
\begin{eqnarray*}
\lefteqn{\Phi_{\Lambda_2} =}\\
& & z_1^4 \bar{z_1}^4+z_2^4 \bar{z_1}^4-z_1^4 z_2^4 \bar{z_1}^8+4
z_1^5 z_2 \bar{z_1}^5 \bar{z_2}-4 z_1 z_2^5 \bar{z_1}^5 \bar{z_2}+12
z_1^2 z_2^2 \bar{z_1}^2 \bar{z_2}^2\\
& & +2 z_1^6 z_2^2 \bar{z_1}^6 \bar{z_2}^2+2
z_1^2 z_2^6 \bar{z_1}^6 \bar{z_2}^2+z_1^4 \bar{z_2}^4+z_2^4 \bar{z_2}^4-z_1^8
\bar{z_1}^4 \bar{z_2}^4-4 z_1^4 z_2^4 \bar{z_1}^4 \bar{z_2}^4\\
& &-z_2^8
\bar{z_1}^4 \bar{z_2}^4-4 z_1^5 z_2 \bar{z_1} \bar{z_2}^5+4 z_1
z_2^5 \bar{z_1} \bar{z_2}^5+2 z_1^6 z_2^2 \bar{z_1}^2 \bar{z_2}^6+2
z_1^2 z_2^6 \bar{z_1}^2 \bar{z_2}^6-z_1^4 z_2^4 \bar{z_2}^8.\\
\end{eqnarray*}

Notice that in contrast to the cyclic case, we get off-diagonal terms, so it is not enough to simply count the number of terms to get the number of eigenvalues.  Rewriting in terms of polynomials invariant under the $Q_2$-action, we get
\begin{eqnarray*}
\lefteqn{\Phi_{\Lambda_2} =}\\
& & (z_1^4+z_2^4)(\bar{z_1}^4 +\bar{z_2}^4)-4 z_1^4 z_2^4 \bar{z_1}^4 \bar{z_2}^4 +12z_1^2 z_2^2 \bar{z_1}^2 \bar{z_2}^2 +4 z_1 z_2 (z_1^4 -z_2^4) \bar{z_1} \bar{z_2} (\bar{z_1}^4 -\bar{z_2}^4)\\
& & +2 z_1^2 z_2^2 (z_1^4 +z_2^4) \bar{z_1}^2 \bar{z_2}^2 (\bar{z_1}^4 +\bar{z_2}^4)-z_1^4 z_2^4 (\bar{z_1}^8+\bar{z_2}^8)-\bar{z_1}^4 \bar{z_2}^4 (z_1^8+z_2^8).\\
\end{eqnarray*}

Equivalently we get $$\Phi_{\Lambda_2} =
\begin{pmatrix}
\bar{z_1}^4 + \bar{z_2}^4\\
\bar{z_1} \bar{z_2} (\bar{z_1}^4 -\bar{z_2}^4)\\
\bar{z_1}^2 \bar{z_2}^2 (\bar{z_1}^4 +\bar{z_2}^4)\\
\bar{z_1}^2 \bar{z_2}^2\\
\bar{z_1}^4 \bar{z_2}^4\\
\bar{z_1}^8+\bar{z_2}^8\\
\end{pmatrix}^{T}
\begin{pmatrix}
1 &  0 & 0 &  0 & 0 & 0\\
0 & 4 & 0 &  0 & 0 & 0\\
0 &  0 & 2 &  0 & 0 & 0\\
0 &  0 & 0 & 12 & 0 & 0\\
0 &  0 & 0 &  0 & -4 & -1\\
0 &  0 & 0 &  0 & -1 & 0\\
\end{pmatrix}
\begin{pmatrix}
z_1^4 + z_2^4\\
z_1 z_2 (z_1^4 -z_2^4)\\
z_1^2 z_2^2 (z_1^4 +z_2^4)\\
z_1^2 z_2^2\\
z_1^4 z_2^4\\
z_1^8+z_2^8\\
\end{pmatrix}.$$

Hence the eigenvalues of $\Phi_{\Lambda_2}$ are 1, 4, 2, 12, $-2+\sqrt{5}$, $-2-\sqrt{5}$.  Therefore $S(\Lambda_2)=(5,1)$.  

For clarity, we explicitly write $\Phi_{\Lambda_2}$ as a difference of squared norms.  Let
$$A(z)=\begin{pmatrix} 
z_1^4+z_2^4\\ 
2 z_1 z_2 (z_1^4-z_2^4)\\ 
\sqrt{2} z_1^2 z_2^2 (z_1^4 +z_2^4)\\ 
\sqrt{12} z_1^2 z_2^2\\ 
(\sqrt{-2 + \sqrt{5}}) (z_1^8 + (2 - \sqrt{5}) z_1^4 z_2^4 + z_2^8\\ 
\end{pmatrix},$$ and $$B(z)=\left( (\sqrt{2 + \sqrt{5}}) (z_1^8 + (2 + \sqrt{5}) z_1^4 z_2^4 + z_2^8)\right).$$  Then we have $\Phi_{\Lambda_2}=||A(z)||^2-||B(z)||^2$.

We proceed along these lines for general $p$.  First we prove a theorem relating $\Phi_{\Lambda_p}$ to the cyclic case.  Since the elements of $\Lambda_p$ split equally into diagonal and anti-diagonal matrices, we can prove a theorem analogous to Theorem \ref{Dihedral} for this case.

\begin{proposition}  \label{Quaternion} The invariant polynomial corresponding to the representation $\eta$ satisfies:
\begin{eqnarray*}\Phi_{\Lambda_p} =  f_{2p,2p-1}(\left|z_1 \right|^2, \left|z_2 \right|^2)+
f_{2p,2p-1}(z_2 \bar{z_1}, -z_1 \bar{z_2})\\
\qquad-f_{2p,2p-1}(\left|z_1 \right|^2,\left|z_2 \right|^2)f_{2p,2p-1}(z_2 \bar{z_1}, -z_1 \bar{z_2}).
\end{eqnarray*}
\end{proposition}
\begin{proof}  As we alluded to above, the key idea is to notice that $$\Lambda_p=\left\{\begin{pmatrix} 
\omega^j & 0\\ 
0 & \omega^{-j}
\end{pmatrix},\begin{pmatrix} 
0 & \omega^{j}\\ 
-\omega^{-j} & 0
\end{pmatrix}|\; j=0, \cdots , 2p-1 \right\}.$$  Then the result follows from the calculation below
\begin{eqnarray*}
\lefteqn{\Phi_{\Lambda_p}(z,\bar{z}) = 1- \prod_{\gamma \in \Lambda_p}{\left( 1-\left< \gamma z, z \right> \right)}}\\
&= & 1- \left( \prod_{j=0}^{2p-1}{\left( 1-\left< \begin{pmatrix} 
\omega^j & 0\\ 
0 & \omega^{-j}
\end{pmatrix} z, z \right> \right)}\right)\left( \prod_{j=0}^{2p-1}{\left( 1-\left< \begin{pmatrix} 
0 & \omega^{j}\\ 
-\omega^{-j} & 0
\end{pmatrix}z, z \right> \right)}\right)\\
&= & 1- \left( \prod_{j=0}^{2p-1}{\left( 1 - \omega^j z_1 \bar{z_1} - \omega^{-j} z_2 \bar{z_2} \right)}\right)\left( \prod_{j=0}^{2p-1}{\left( 1 - \omega^j z_2 \bar{z_1} + \omega^{-j} z_1 \bar{z_2} \right)}\right)\\
& = & 1- \left( 1- f_{2p,2p-1}(|z_1|^2, |z_2|^2)\right)\left( 1-f_{2p,2p-1}(z_2 \bar{z_1}, -z_1\bar{z_2})\right)\\
&= & f_{2p,2p-1}(\left|z_1 \right|^2, \left|z_2 \right|^2)+
f_{2p,2p-1}(z_2 \bar{z_1}, -z_1 \bar{z_2})\\
& & \qquad-f_{2p,2p-1}(\left|z_1 \right|^2,\left|z_2 \right|^2)f_{2p,2p-1}(z_2 \bar{z_1}, -z_1 \bar{z_2}).
\end{eqnarray*}
\end{proof}

Now we proceed by using the previous theorem to express the polynomial $\Phi_{\Lambda_p}$ in terms of the $f_{2p,2p-1}$ polynomials.  The goal then is to express $\Phi_{\Lambda_p}$ in terms of the following linearly independent $Q_p$-invariant polynomials: $$z_1^{2p} + z_2^{2p},\; z_1^j z_2^j (z_1^{2p}+ (-1)^j z_2^{2p}),\; (z_1 z_2)^{2j}, z_1^{4p}+z_2^{4p}$$ for $j=1, \cdots, \, p$.

For the reader's convenience, we again recall $$f_{2p,2p-1}(x,y)= x^{2p}+y^{2p}+\sum_{j=1}^{p}{(-1)^{j-1}c_{2p,j}(xy)^j}.$$  Then by Proposition \ref{Quaternion} 
\begin{eqnarray*}
\lefteqn{\Phi_{\Lambda_p}(z, \bar{z}) = } \\
& & (z_1 \bar{z_1})^{2p} + (z_2 \bar{z_2})^{2p} + \sum_{j=1}^{p}{(-1)^{j-1}c_{2p,j}(z_1 z_2 \bar{z_1} \bar{z_2})^j}+(z_2 \bar{z_1})^{2p}+(z_1 \bar{z_2})^{2p}\\
& &+ \sum_{j=1}^{p}{(-1) c_{2p,j} (z_1 z_2 \bar{z_1} \bar{z_2})^j}-z_1^{2p}z_2^{2p}\bar{z_1}^{4p}-z_1^{4p}\bar{z_1}^{2p}\bar{z_2}^{2p}-z_2^{4p}\bar{z_1}^{2p}\bar{z_2}^{2p}\\
& &-z_1^{2p}z_2^{2p}\bar{z_2}^{4p}+\sum_{j=1}^{p}{c_{2p,j}z_1^{2p+j}z_2^j\bar{z_1}^{2p+j}\bar{z_2}^j}+\sum_{j=1}^{p}{c_{2p,j}z_1^{j}z_2^{2p+j}\bar{z_1}^{j}\bar{z_2}^{2p+j}}\\
& &+\sum_{j=1}^{p}{(-1)^{j}c_{2p,j}z_1^{j}z_2^{2p+j}\bar{z_1}^{2p+j}\bar{z_2}^j}+\sum_{j=1}^{p}{(-1)^{j}c_{2p,j}z_1^{2p+j}z_2^j\bar{z_1}^{j}\bar{z_2}^{2p+j}}\\
& & -\left( \sum_{j=1}^{p}{(-1)^{j-1}c_{2p,j}(z_1 z_2 \bar{z_1} \bar{z_2})^j} \right) \left( \sum_{j=1}^{p}{(-1) c_{2p,j} (z_1 z_2 \bar{z_1} \bar{z_2})^j} \right).\\
\end{eqnarray*}

Notice that all the odd power terms drop from the product, and we get the following simplification.
\begin{eqnarray*}
\lefteqn{\left( \sum_{j=1}^{p}{(-1)^{j-1}c_{2p,j}(z_1 z_2 \bar{z_1} \bar{z_2})^j} \right) \left( \sum_{j=1}^{p}{(-1) c_{2p,j} (z_1 z_2 \bar{z_1} \bar{z_2})^j} \right) = } \\ 
& &\sum_{j=1}^{p}{\left( 2 \sum_{k=j+1}^{\text{min}(2j-1,p)}{(-1)^{k-1} c_{2p,k}c_{2p,2j-k}} +(-1)^{j-1}c_{2p,j}^2\right)(z_1 z_2 \bar{z_1} \bar{z_2})^{2j}}.
\end{eqnarray*}

With some additional effort, we write the previous expression in terms of the invariant polynomials given above,
\begin{eqnarray*}
\lefteqn{\Phi_{\Lambda_p}(z, \bar{z}) = }\\
& & (z_1^{2p}+z_2^{2p})(\bar{z_1}^{2p}+\bar{z_2}^{2p})+\sum_{j=1}^{p}{c_{2p,i}z_1^j z_2^i(z_1^{2p}+(-1)^j z_2^{2p})\bar{z_1}^j \bar{z_2}^j(\bar{z_1}^{2p}+(-1)^j \bar{z_2}^{2p})}\\
& & + \sum_{j=1}^{\left \lfloor \frac{p}{2} \right \rfloor}{\left( 2 \sum_{k=j+1}^{2j-1}{(-1)^{k-1} c_{2p,k}c_{2p,2j-k}} +(-1)^{j-1}c_{2p,j}^2-2c_{2p,2j}\right) (z_1 z_2 \bar{z_1} \bar{z_2})^{2j}}\\
& & + \sum_{j=\left \lfloor \frac{p}{2}\right \rfloor+1}^{p}{\left( 2 \sum_{k=j+1}^{p}{(-1)^{k-1} c_{2p,k}c_{2p,2j-k}} +(-1)^{j-1}c_{2p,j}^2\right) (z_1 z_2 \bar{z_1} \bar{z_2})^{2j}}\\
& & - z_1^{2p}z_2^{2p}(\bar{z_1}^{4p}+\bar{z_2}^{4p})-\bar{z_1}^{2p}\bar{z_2}^{2p}(z_1^{4p}+z_2^{4p}).\\
\end{eqnarray*}

The goal now is to determine the signs of the coefficients in the above expression.  First we take care of the obvious cases.  The coefficient of $|z_1^j z_2^j (z_1^{2p}+ (-1)^j z_2^{2p})|^2$ is $c_{2p,j}$ which is positive.  The coefficient of $|z_1^{2p}+z_2^{2p}|^2$ is positive.  We now consider the coefficient of $|z_1 z_2|^{4j}$.

\begin{definition}
Define $d_{k}$ to be the coefficient of $(z_1 \overline{z_1} z_2 \overline{z_2})^{2k}$ in $\Phi_{\Lambda_p}(z, \bar{z})$.  Further we define the polynomial $D_p$ by $$D_{p}(t)=\sum_{j=1}^p{d_{k}t^{2k}}.$$
\end{definition}

We give an exact formula for $D_p(t)$ in the next proposition.

\begin{proposition}  The polynomial $D_{p}(t)$ is given by 
\begin{eqnarray*}
D_{p}(t)=1&-&\frac{1}{4^p}\big( \left(1+ a + b + a b\right)^{2p}+ \left(1- a + b - a b\right)^{2p}\\
&+&\left(1+ a - b - a b\right)^{2p}+ \left(1- a - b + a b\right)^{2p}\big)
\end{eqnarray*} where $a=\sqrt{1-4t}$ and $b=\sqrt{1+4t}$.
\end{proposition}
\begin{proof}
By Proposition \ref{Quaternion} and Theorem \ref{thm:cyclicproposition}, 
\begin{eqnarray*}
\lefteqn{\Phi_{\Lambda_p} =}\\
& & 1+|z_1|^{4p} +|z_2|^{4p}-\left( \frac{1+\sqrt{1-4 |z_1 z_2|^2}}{2} \right)^{2p} - \left( \frac{1-\sqrt{1-4|z_1 z_2|^2}}{2}\right)^{2p}\\
&+&1+(z_2 \bar{z_1})^{2p} +(z_1 \bar{z_2})^{2p}-\left( \frac{1+\sqrt{1+4 |z_1 z_2|^2}}{2} \right)^{2p} - \left( \frac{1-\sqrt{1+4|z_1 z_2|^2}}{2}\right)^{2p}\\
&-&\left(1+|z_1|^{4p} +|z_2|^{4p}-\left( \frac{1+\sqrt{1-4 |z_1 z_2|^2}}{2} \right)^{2p} - \left( \frac{1-\sqrt{1-4|z_1 z_2|^2}}{2}\right)^{2p}\right)\\
& \times&\left(1+(z_2 \bar{z_1})^{2p} +(z_1 \bar{z_2})^{2p}-\left( \frac{1+\sqrt{1+4 |z_1 z_2|^2}}{2} \right)^{2p} - \left( \frac{1-\sqrt{1+4|z_1 z_2|^2}}{2}\right)^{2p}\right).
\end{eqnarray*}

Next we let $t=|z_1 z_2|^2$.  Take all the terms involving $t$ in the previous expression to get the following:
\begin{eqnarray*}
D_p(t)& =& 1-\left( \frac{1+\sqrt{1-4 t}}{2} \right)^{2p} - \left( \frac{1-\sqrt{1-4t}}{2}\right)^{2p}\\
&+&1-\left( \frac{1+\sqrt{1+4t}}{2} \right)^{2p} - \left( \frac{1-\sqrt{1+4t}}{2}\right)^{2p}\\
&-&\left(1-\left( \frac{1+\sqrt{1-4t}}{2} \right)^{2p} - \left( \frac{1-\sqrt{1-4t}}{2}\right)^{2p}\right)\\
&\times &\left(1-\left( \frac{1+\sqrt{1+4 t}}{2} \right)^{2p} - \left( \frac{1-\sqrt{1+4t}}{2}\right)^{2p}\right).\\
\end{eqnarray*}
Multiply this expression out to get the desired result:
\begin{eqnarray*} 
D_p(t) & = & 1- \frac{1}{4^p}\Bigg( \left( 1+\sqrt{1+4 t} \right)^{2p}\left( 1+\sqrt{1-4 t} \right)^{2p} + \left( 1+\sqrt{1-4 t} \right)^{2p}\left( 1-\sqrt{1+4 t} \right)^{2p}\\
& +& \left( 1+\sqrt{1+4 t} \right)^{2p}\left( 1-\sqrt{1-4 t} \right)^{2p}+ \left( 1-\sqrt{1-4 t} \right)^{2p}\left( 1-\sqrt{1+4 t} \right)^{2p} \Bigg).
\end{eqnarray*}
\end{proof}

\begin{lemma}  \label{l:product}  The following identity holds:
\begin{align*}  
&\left(1+ a + b + a b\right)^{2p}+ \left(1- a + b - a b\right)^{2p}+\left(1+ a - b - a b\right)^{2p}+ \left(1- a - b + a b\right)^{2p}\\
&=4 \sum_{j=0}^{p}{\sum_{k=0}^p{\binom{2p}{2j}\binom{2p}{2k}a^{2j}b^{2k}}}\\
\end{align*}
\end{lemma}
\begin{proof}  We begin by factoring and using the binomial theorem.
\begin{align*}  
&\left(1+ a + b + a b\right)^{2p}+ \left(1- a + b - a b\right)^{2p}+\left(1+ a - b - a b\right)^{2p}+ \left(1- a - b + a b\right)^{2p}\\
&=\left((1+ a)(1+b)\right)^{2p}+ \left((1- a)(1+ b)\right)^{2p}+\left((1+ a)(1-b)\right)^{2p}+ \left((1- a)(1-b)\right)^{2p}\\
&=\left( \sum_{j=0}^{2p}{\binom{2p}{j}a^{j}}\right)\left( \sum_{k=0}^{2p}{\binom{2p}{k}b^{k}}\right)+
\left( \sum_{j=0}^{2p}{\binom{2p}{j}(-1)^{j}a^{j}}\right)\left( \sum_{k=0}^{2p}{\binom{2p}{k}b^{k}}\right)\\
&+\left( \sum_{j=0}^{2p}{\binom{2p}{j}a^{j}}\right)\left( \sum_{k=0}^{2p}{\binom{2p}{k}(-1)^{k}b^{k}}\right)+
\left( \sum_{j=0}^{2p}{\binom{2p}{j}(-1)^{j}a^{j}}\right)\left( \sum_{k=0}^{2p}{\binom{2p}{k}(-1)^{k}b^{k}}\right).\\
\end{align*}
After multiplying out the right hand side and collecting terms we get
\begin{align*}
\sum_{j=0}^{2p}{\sum_{k=0}^{2p}{\binom{2p}{j}\binom{2p}{k}a^j b^k \left( 1+(-1)^j+(-1)^k+(-1)^{j+k}\right)}}.\\
\end{align*}

\noindent Since $$\left( 1+(-1)^j+(-1)^k+(-1)^{j+k}\right)=\begin{cases}4 & \text{if $j$ and $k$ are both even,}\\
0 & \text{otherwise,}\\\end{cases}$$  the identity follows after reindexing.
\end{proof}

Let $a=\sqrt{z}$ and $b=\sqrt{\bar{z}}$ in the Lemma \ref{l:product}, then we have the following:
$$4 \sum_{j=0}^{p}{\sum_{k=0}^p{\binom{2p}{2j}\binom{2p}{2k}z^{j}\bar{z}^{k}}}=\left| 2\sum_{k=0}^p{\binom{2p}{2k}z^k} \right|^2.$$

\begin{lemma}  Given a polynomial $p(x+i y)$ with all negative real roots, then $\left| p(x+i y) \right|^2$ has positive coefficients.
\end{lemma}
\begin{proof}  Let $a_0$, $\cdots$, $a_d$ be the absolute values of the roots of $p$.  Then expanding and simplifying $p$, we get 
\begin{eqnarray*}
\left| p(x+i y) \right|^2 &= &\left| \prod_{j=0}^d{x+iy+a_j}\right|^2\\
&=&\prod_{j=0}^d{\left(x+iy+a_j\right)\left(x-iy+a_j\right)}=\prod_{j=0}^d{\left( x^2+2x a_j+y^2+a_j^2\right)}.\\
\end{eqnarray*}
  In the last expression only positive real coefficients occur, so after expanding the product we get the desired result.
\end{proof}

D'Angelo provided me the statement and proof of the following lemma.

\begin{lemma}  The following identity holds: $$P(z)=2\sum_{k=0}^{p}{\binom{2p}{2k}z^k}=\prod_{j=0}^{p-1}{\left(z+\tan^2{\left(\frac{(2j+1)\pi}{4p}\right)}\right)},$$ and hence all the roots of $P$ are negative.
\end{lemma}
\begin{proof}  Taking proper care of the choice of square root, we can rewrite the given polynomial in the following way: $$P(z)=2\sum_{k=0}^{p}{\binom{2p}{2k}z^k}=\left( 1-\sqrt{z} \right)^{2p}+\left( 1+\sqrt{z} \right)^{2p}.$$

Setting the right hand side equal to zero yields
\begin{equation*}
\left( \frac{1-\sqrt{z}}{1+\sqrt{z}}\right)^{2p}=-1.
\end{equation*}

Taking $2p$-th roots gives
\begin{equation*}
 \frac{1-\sqrt{z}}{1+\sqrt{z}}= e^{\frac{(2n+1)\pi i}{2p}}
\end{equation*}
for $n=0, \cdots, 2p-1$.

We solve for $\sqrt{z}$:
\begin{align*}
\sqrt{z} &= \frac{1-e^{\frac{(2n+1)\pi i}{2p}}}{1+e^{\frac{(2n+1)\pi i}{2p}}}
= \frac{e^{-\frac{(2n+1)\pi i}{4p}}-e^{\frac{(2n+1)\pi i}{4p}}}{e^{-\frac{(2n+1)\pi i}{4p}}+e^{\frac{(2n+1)\pi i}{4p}}}= i \tan{\left( \frac{(2n+1)\pi}{4p}\right)}.
\end{align*}

The roots of $P$ are therefore $-\tan^2{\left(\frac{(2j+1)\pi}{4p}\right)}$, and hence the identity follows.
\end{proof}

Finally we combine the previous lemmas to determine the sign of $d_{k}$.

\begin{proposition}  \label{proposition:dksign} For $1 \leq k \leq p$,
\begin{enumerate}
\item $d_{k} > 0$ for $k$ odd.
\item $d_{k} < 0$ for $k$ even. 
\end{enumerate}
\end{proposition}
\begin{proof}  We have the following relationship between $D_{p}(t)$ and $P(z)$:  $$D(i t)= 1-\frac{1}{4^p} P(1+4 i t).$$  By the previous lemmas $P(z)$ has all positive coefficients, thus $D_{p}(it)$ must have all negative coefficients.  Since $D(t)$ is a polynomial with only even powers, the transformation $t \mapsto i t$ changes the sign of $d_{k}$ for odd $k$ and does not change the sign of $d_{k}$ when $k$ is even.  Therefore $d_{k}$ must be positive for $k$ odd and negative for $k$ even.
\end{proof}

We rephrase the results in terms of matrices: $$\Phi_{\Lambda_p}(z, \bar{z})= d^{*}M_p d$$ where 
$$d=\begin{pmatrix}
z_1^{2p}+z_2^{2p}\\
z_1 z_2(z_1^{2p}+(-1) z_2^{2p})\\
\vdots\\
z_1^p z_2^p(z_1^{2p}+(-1)^p z_2^{2p})\\
z_1^{2} z_2^{2}\\
\vdots\\
z_1^{2p} z_2^{2p}\\
z_1^{4p}+z_2^{4p}\\
\end{pmatrix},$$ and $$M_p=\begin{pmatrix}
1 & 0   & 0 & 0\\
0 & E_{p,1} & 0 & 0\\
0 & 0 & E_{p,2} & 0\\
0 & 0 & 0   & E_{p,3}\\
\end{pmatrix}$$ where $E_{p,1}$ is the $p$ by $p$ matrix with $c_{2p,j}$ on the diagonal.  Also $E_{p,2}$ is the square matrix of size $p-1$ with diagonal entries $$(E_{p,2})_{jj}=2 \sum_{k=j+1}^{2j-1}{(-1)^{k-1} c_{2p,k}c_{2p,2j-k}} +(-1)^{j-1}c_{2p,j}^2-2c_{2p,2j}$$  for $1 \leq j \leq \left \lfloor \frac{p}{2} \right \rfloor$, and $$(E_{p,2})_{jj}=2 \sum_{k=j+1}^{p}{(-1)^{k-1} c_{2p,k}c_{2p,2j-k}} +(-1)^{j-1}c_{2p,j}^2$$  for $\left \lfloor \frac{p}{2} \right \rfloor < j \leq p-1$.   Finally we have the 2 by 2 matrix $$E_{p,3}=\begin{pmatrix} (-1)^{p-1}c_{2p,p}^2 & -1 \\ -1 &0\\ \end{pmatrix}.$$

Now we are left with the task of computing the signature pair of $M_p$.  We proceed by counting the number of eigenvalues of each sign in the submatrices.

\begin{proposition}  For $1 \leq j \leq p-1$, 
\begin{enumerate}
\item $(E_{p,2})_{jj}>0$ if $j$ is odd.
\item $(E_{p,2})_{jj}<0$ if $j$ is even.
\end{enumerate}
\end{proposition}
\begin{proof}  Follows from Proposition \ref{proposition:dksign}.
\end{proof}

The diagonal matrices $E_{p,1}$ and $E_{p,2}$ have non-zero diagonal entries.  The submatrix $E_{p,1}$ has $p$ eigenvalues, all of which are positive.  Moreover, by the proposition, the diagonal entries in the matrices $E_{p,2}$ alternate sign.  Also the matrix $E_{p,3}$ has one eigenvalue of each sign.  Thus combining these results for the submatrices of $M_p$, we obtain one of our main results.

\begin{theorem}  The signature pair of the binary dihedral group with $4p$ elements is given by $$S(\Lambda_p)=\left(2+p+\left \lfloor \frac{p}{2} \right \rfloor, 1+\left \lfloor \frac{p-1}{2} \right \rfloor \right).$$
\end{theorem}

Taking the limit as $p$ goes to infinity we get the following theorem.

\begin{theorem} The asymptotic positivity ratio for $\Lambda_p$ is $\frac{3}{4}$.
\end{theorem}
 
\subsection{Binary Tetrahedral Group}

The binary tetrahedral group is given by $$T:=\left< a, b \;   | \; a^3=b^3=(a b)^2\right>$$  and has order 24.  We represent $T$ in $SU(2)$ using the Springer description \cite{Sp77}.  Let 
\begin{equation} \label{tetrahedralgens}
r = 
\begin{pmatrix} 
\epsilon & 0\\ 
0 & \epsilon^{-1}
\end{pmatrix} \; \;
s = 
\begin{pmatrix} 
0 & 1\\ 
-1 & 0
\end{pmatrix} \; \;
t = 
\frac{1}{\sqrt{2}}\begin{pmatrix} 
\epsilon^{-1} & \epsilon^{-1}\\
-\epsilon & \epsilon
\end{pmatrix}
\end{equation}
where $\epsilon=e^{\frac{\pi i}{4}}$.

Define the faithful unitary representation $\kappa : T \to SU(2)$ by $$\kappa(a)= st^{-1} \;\text{ and } \; \kappa(b)= t.$$  
Let $\Gamma = \kappa(T)$.  For all $\gamma \in \Gamma$ we can represent $\gamma$ in the following way:$$\gamma= r^{2j} s^k t^l$$ for some $0 \leq j < 3$, $0 \leq k < 2$, and $0\leq l < 3$.  We remark that all faithful representations of $T$ in $SU(2)$ are equivalent to the representation given by $\kappa$.

We express $\Phi_\Gamma$ in terms in $\Gamma$-invariant polynomials.    The following 14 linearly independent $\Gamma$-invariant polynomials appear:

\begin{tabular}{lcl}
$C_1$ & $=$ & $z_1^{16}+28 z_1^{12} z_2^4+198 z_1^8 z_2^8+28 z_1^4 z_2^{12}+z_2^{16}$ \\[2pt]
$C_2$ & $=$ & $z_1^{20}-19 z_1^{16} z_2^4-494 z_1^{12} z_2^8-494 z_1^8 z_2^{12}-19 z_1^4 z_2^{16}+z_2^{20}$
\\[2pt]
$C_3$ & $=$ & $z_1^{18} z_2^2+12 z_1^{14} z_2^6-26 z_1^{10} z_2^{10}+12 z_1^6 z_2^{14}+z_1^2 z_2^{18}$
\\[2pt]
$C_4$ & $=$ & $-z_1^{21} z_2-27 z_1^{17} z_2^5-170 z_1^{13} z_2^9+170 z_1^9 z_2^{13}+27 z_1^5 z_2^{17}+z_1
z_2^{21}$ \\[2pt]
$C_5$ & $=$ & $-z_1^{15} z_2^3+3 z_1^{11} z_2^7-3 z_1^7 z_2^{11}+z_1^3 z_2^{15}$ \\[2pt]
$C_6$ & $=$ & $z_1^{12}-33 z_1^8 z_2^4-33 z_1^4 z_2^8+z_2^{12}$ \\[2pt]
$C_7$ & $=$ & $-z_1^{13} z_2-13 z_1^9 z_2^5+13 z_1^5 z_2^9+z_1 z_2^{13}$ \\[2pt]
$C_8$ & $=$ & $-z_1^{17} z_2+34 z_1^{13} z_2^5-34 z_1^5 z_2^{13}+z_1 z_2^{17}$ \\[2pt]
$C_9$ & $=$ & $z_1^8+14 z_1^4 z_2^4+z_2^8$ \\[2pt]
$C_{10}$ & $=$ & $z_1^{24}+\left(-\frac{4692}{35}+\frac{1}{35} \left(2382+\sqrt{119948010}\right)\right) z_1^{20} z_2^4$\\[2pt]
 & $+$& $\left(\frac{45333}{35}+\frac{4}{35}\left(-2382-\sqrt{119948010}\right)\right) z_1^{16} z_2^8$\\[2pt]
& $+$ &$\left(\frac{62008}{35}-\frac{6}{35} \left(-2382-\sqrt{119948010}\right)\right)z_1^{12} z_2^{12}$\\[2pt]
& $+$& $\left(\frac{45333}{35}+\frac{4}{35} \left(-2382-\sqrt{119948010}\right)\right) z_1^8 z_2^{16}$\\[2pt]
& $+$& $\left(-\frac{4692}{35}+\frac{1}{35}\left(2382+\sqrt{119948010}\right)\right) z_1^4 z_2^{20}+z_2^{24}$ \\[2pt]
$C_{11}$ & $=$ & $z_1^{10} z_2^2-2 z_1^6 z_2^6+z_1^2 z_2^{10}$ \\[2pt]
$C_{12}$ & $=$ & $z_1^{24}+\left(-\frac{4692}{35}+\frac{1}{35} \left(2382-\sqrt{119948010}\right)\right) z_1^{20} z_2^4$\\[2pt]
&$+$& $\left(\frac{45333}{35}+\frac{4}{35}\left(-2382+\sqrt{119948010}\right)\right) z_1^{16} z_2^8$\\[2pt]
&$+$&$\left(\frac{62008}{35}-\frac{6}{35} \left(-2382+\sqrt{119948010}\right)\right)z_1^{12} z_2^{12}$\\[2pt]
&$+$&$\left(\frac{45333}{35}+\frac{4}{35} \left(-2382+\sqrt{119948010}\right)\right) z_1^8 z_2^{16}$\\[2pt]
& $+$&$\left(-\frac{4692}{35}+\frac{1}{35}
\left(2382-\sqrt{119948010}\right)\right) z_1^4 z_2^{20}+z_2^{24}$ \\[2pt]
$C_{13}$ & $=$ & $z_1^{22} z_2^2-35 z_1^{18} z_2^6+34 z_1^{14} z_2^{10}+34 z_1^{10} z_2^{14}-35 z_1^6 z_2^{18}+z_1^2
z_2^{22}$ \\[2pt]
$C_{14}$ & $=$ & $-z_1^5 z_2+z_1 z_2^5$.\\
\end{tabular}

Define
$$A(z)=\begin{pmatrix}
\sqrt{\frac{305805}{128}} C_1\\ 
\sqrt{\frac{122199}{64}} C_2\\
\sqrt{\frac{14815}{16}} C_4\\
\sqrt{740} C_5\\ 
\sqrt{\frac{2725}{4}} C_6\\ 
\sqrt{\frac{495}{4}} C_9\\ 
\sqrt{\frac{1}{128} (-2382 + \sqrt{119948010})} C_{12}\\
\sqrt{\frac{1191}{32}} C_{13}\\ 
\sqrt{24} C_{14}
\end{pmatrix}$$ and $$B(z)=\begin{pmatrix}
\sqrt{\frac{48783}{32}}C_3\\
\sqrt{680} C_7\\
\sqrt{\frac{1157}{2}} C_8\\ 
\sqrt{\frac{1}{128} (2382 + \sqrt{119948010})} C_{10}\\ 
\sqrt{\frac{171}{2}} C_{11}
\end{pmatrix}.$$  

Using Mathematica \cite{Mathematica} one can verify that $\Phi_\Gamma$ decomposes in the following way $$\Phi_{\Gamma}=||A(z)||^2-||B(z)||^2.$$  Therefore $S(\Gamma)=(9,5)$.

\begin{remark}  If $\Gamma<SU(2)$, and $\Gamma$ is isomorphic to $T$, then $S(\Gamma)=\left( 9,5 \right)$.
\end{remark}

\subsection{Binary Octahedral Group}

The binary octahedral group is given by $$O:=\left< a, b \;   | \; a^4=b^3=(a b)^2\right>$$  and has order 48.  We again represent $O$ in $SU(2)$ using the Springer description \cite{Sp77}.  Recall the generators of the binary tetrahedral group $r$, $s$, and $t$ given above in \eqref{tetrahedralgens}.  Let $\tau : O \to SU(2)$ be a faithful unitary representation generated by $$\tau(a)= rt \;\text{ and } \; \tau(b)= t.$$  Notice that $$(rt)^4=t^3=(rt^2)^2=-1.$$  

Let $\Gamma = \tau(O)$. For all $\gamma\in \Gamma$ we can represent $\gamma$ in the following way:$$\gamma= r^{j} s^k t^l$$ for some $0 \leq j < 8$, $0 \leq k < 2$, and $0\leq l < 3$.  We remark that all faithful representations of $O$ in $SU(2)$ are equivalent to the representation given by $\tau$.

The $\Gamma$-invariant polynomial $\Phi_\Gamma$ has 1143 terms.  Using Mathematica we decompose $$\Phi_\Gamma=d^{*}Md$$ where $M$ is the Hermitian coefficient matrix and $d$ is the vector of 135 monomials that appear in $\Phi_\Gamma$.  Again using Mathematica we find that $M$ has rank 26 with 17 positive eigenvalues and 9 negative eigenvalues.  

\begin{remark}  Let $\Gamma<SU(2)$ such that $\Gamma$ is isomorphic to $O$, then $$S(\Gamma)=\left( 17,9 \right).$$
\end{remark}

\subsection{Binary Icosahedral Group}

The binary icosahedral group is given by $$I:=\left< a, b \;   | \; a^5=b^3=(a b)^2\right>$$  and has order 120.  We again represent $I$ in $SU(2)$ using the Springer description \cite{Sp77}.  Let
\begin{equation} \label{icosahedralgens}
r = 
-\begin{pmatrix} 
\epsilon^3 & 0\\ 
0 & \epsilon^{2}
\end{pmatrix} \; \;
s = 
\begin{pmatrix} 
0 & 1\\ 
-1 & 0
\end{pmatrix} \; \;
t = 
\frac{1}{\epsilon^2-\epsilon^{-2}}\begin{pmatrix} 
\epsilon+\epsilon^{-1} & 1\\
1 & -\epsilon-\epsilon^{-1}
\end{pmatrix}
\end{equation}
where $\epsilon=e^{\frac{2\pi i}{5}}$.  Then we define a representation of $I$ in $SU(2)$ by $a=r$ and $b=r^4ts$.  Notice that $$(r)^5=(r^4ts)^3=(r^5ts)^2=-1.$$  In \cite{Sp77}, Springer describes the 120 elements in the binary icosahedral group as follows:  $$\Gamma = \left\{r^h, s r^h, r^ht r^{j},r^h t s r^j | 0\leq h <10, 0 \leq j<5 \right\}.$$ 

The invariant polynomial $\Phi_\Gamma$ has about 500,000 terms in this case.  Using Mathematica we decompose $$\Phi_\Gamma=d^{*}Md$$ where $M$ is the Hermitian coefficient matrix and $d$ is the vector of monomials that appear in $\Phi_\Gamma$.  Again using Mathematica we find that $M$ has rank 62 with 40 positive eigenvalues and 22 negative eigenvalues.  

\begin{remark}  Let $\Gamma<SU(2)$ such that $\Gamma$ is isomorphic to $I$, then $$S(\Gamma)=\left( 40,22 \right).$$
\end{remark}

\section{The Asymptotic Positivity Ratio for the Cyclic Case}
In this section we study the signs of the coefficients of the polynomial $f_{p,q}$.  Since $\Gamma(p,q)$ is a diagonal subgroup, the sign of a coefficient of $f_{p,q}$ corresponds to the sign of an eigenvalue of the underlying matrix of $\Phi_{\Gamma(p,q)}$.  When $q=1$ or $q=2$, we know the exact numbers of positive and negative coefficients.  For general $q$ however, it becomes difficult to determine these numbers exactly.  Instead, we find upper and lower bounds for the number of coefficients of each sign.  We use these bounds to compute the asymptotic positivity ratio as a rational function of $q$.  Then we take the limit as $q$ goes to infinity to show that for the $\Gamma(p,q)$ the asymptotic positivity ratio is $\frac{3}{4}$.

Suppose $$f_{p,q}=\sum_{0\leq r,s \leq p}{c_{r,s}x^r y^s}.$$  Since $f_{p,q}$ is $\Gamma(p,q)$-invariant and the degree is at most $p$, we have $r+qs=kp$ for some $k\in \{1,\cdots,q \}$ and $r+s\leq p$.  In \cite{D2} D'Angelo shows that $c_{r,s}$ is a non-zero integer whenever $x^r y^s$ is an invariant monomial, so the above question translates to determining the number of non-negative integer solutions to the equations $r+qs=kp$ for $k\in \{1,\cdots,q \}$ when $0<r+s \leq p$.  For clarity we formally introduce the following notation.

\begin{notation}  The number of weight $k$ terms in $f_{p,q}$ is $N_{k}(\Gamma(p,q))$.  Denote the number of terms of odd weight by $N_{odd}(\Gamma(p,q))$, and the number of terms of even weight by $N_{even}(\Gamma(p,q))$.
\end{notation}

Also notice that the number of terms of $f_{p,q}$ is the same as the number of eigenvalues since we are using the diagonally generated cyclic group $\Gamma(p,q)$.

The following two lemmas estimate the number of terms of each weight and the total number of terms.
\begin{lemma}  \label{lemmaNWt} The following inequality holds: $$\left|N_k(\Gamma(p,q))-\frac{q-k}{q-1}\cdot\frac{p}{q}\right| \leq 1. $$
\end{lemma}
\begin{proof}  Fix $p$ and $q$.  We want to count the number of non-negative integer solutions $(r,s)$ such that $r+qs=kp$ and $r+s\leq p$ where $1\leq k\leq q$.  Notice that $r=kp-qs$, so $r$ is an integer whenever $s$ is an integer.  Further notice that the two lines $r+qs=kp$ and $r+s = p$ intersect at the point $\left( \frac{p(q-k)}{q-1},\frac{(k-1)p}{q-1}\right)$.  Projecting onto the $s$ coordinate, we observe that $N_k(p,q)$ is equal to the number of integers $s$ such that $\frac{(k-1)p}{q-1} \leq s \leq \frac{kp}{q}$.  Thus $N_k(\Gamma(p,q))$ is within 1 of $\left \lfloor \frac{kp}{q} - \frac{(k-1)p}{q-1} \right \rfloor=\left \lfloor \frac{q-k}{q-1}\cdot\frac{p}{q} \right \rfloor$.
\end{proof}
\begin{remark}  For $k=1$ the total number of solutions $N_1(\Gamma(p,q))$ is $\left \lfloor \frac{p}{q}\right \rfloor+1$.\\
For $k=q$ we have $$ N_q(\Gamma(p,q))=1.$$
\end{remark}
\begin{lemma} \label{lemmaN}  The following inequality holds: $$\left| N(\Gamma(p,q)) - \frac{p}{2} \right| \leq q.$$
\end{lemma}  
\begin{proof}  By Lemma \ref{lemmaNWt}, $$\frac{q-k}{q-1}\cdot\frac{p}{q}-1\leq N_k(\Gamma(p,q))\leq \frac{q-k}{q-1}\cdot\frac{p}{q}+ 1,$$ and by definition $N(\Gamma(p,q))=\sum_{k=1}^{q}{N_k(\Gamma(p,q))}$.  Therefore applying Lemma \ref{lemmaNWt} $q$ times yields $$\sum_{k=1}^{q}{\frac{q-k}{q-1}\cdot\frac{p}{q}}-q\leq N(\Gamma(p,q)) \leq \sum_{k=1}^{q}{\frac{q-k}{q-1}\cdot\frac{p}{q}}+q.$$  Factoring and rearranging the sum gives 
\begin{eqnarray*}
\sum_{k=1}^{q}{\frac{q-k}{q-1}\cdot\frac{p}{q}} & = & \frac{p}{q(q-1)}\sum_{k=1}^{q}{(q-k)}\\
& = & \frac{p}{q(q-1)}(q^2-\frac{q(q+1)}{2}) = \frac{p}{2}.
\end{eqnarray*}
Thus combining the last two calculations gives the result $$\frac{p}{2}-q\leq N(\Gamma(p,q)) \leq \frac{p}{2}+q.$$
\end{proof}

Next we show in the limit that the ratio of the number of terms of odd weight to the total number of terms equals the ratio of the number of terms of even weight to the total number of terms.

\begin{lemma}  \label{lemma2}  The following limit holds: $$\lim_{q\to\infty}{\lim_{p\to \infty}{\frac{N_{odd}(\Gamma(p,q))}{N(\Gamma(p,q))}}}=\lim_{q\to\infty}{\lim_{p\to \infty}{\frac{N_{even}(\Gamma(p,q))}{N(\Gamma(p,q))}}}=\frac{1}{2}$$
\end{lemma}
\begin{proof}  There are four similar cases depending on the residue of $q$ modulo 4.  We consider the case where $q=4r$.  Recall $$N_{odd}(\Gamma(p,q))=\sum_{k=1}^{2r}{N_{2k-1}(\Gamma(p,q))}.$$  By Lemma \ref{lemmaNWt} 
\begin{eqnarray*}
\sum_{k=1}^{2r}{\frac{q-(2k-1)}{q-1}\cdot\frac{p}{q}}-2r & \leq & N_{odd} \leq \sum_{k=1}^{2r}{\frac{q-(2k-1)}{q-1}\cdot\frac{p}{q}}+2r.
\end{eqnarray*}
Rearranging and simplifying yields
\begin{eqnarray*}
\frac{p}{q-1}\cdot\frac{q}{4}-\frac{q}{2} & \leq & N_{odd} \leq \frac{p}{q-1}\cdot\frac{q}{4}+\frac{q}{2}.
\end{eqnarray*}
Next apply Lemma \ref{lemmaN} to get
\begin{eqnarray*}
\frac{\frac{p}{q-1}\cdot\frac{q}{4}-\frac{q}{2}}{\frac{p}{2}+\frac{q}{2}} & \leq & \frac{N_{odd}}{N} \leq \frac{\frac{p}{q-1}\cdot\frac{q}{4}+\frac{q}{2}}{\frac{p}{2}-\frac{q}{2}}\\
\frac{pq - 2q(q-1)}{2(p+q)(q-1)} &\leq& \frac{N_{odd}}{N} \leq \frac{pq + 2q(q-1)}{2(p-q)(q-1)}.
\end{eqnarray*}
Take limit as $p \to \infty$
\begin{eqnarray*}
\frac{q}{2(q-1)} &\leq& \lim_{p\to \infty}{\frac{N_{odd}}{N}} \leq \frac{q}{2(q-1)}.
\end{eqnarray*}
Thus 
\begin{eqnarray*}
\lim_{p\to \infty}{\frac{N_{odd}}{N}} = \frac{q}{2(q-1)}.
\end{eqnarray*}  

We proceed similarly for the even case.  First $$N_{even}(\Gamma(p,q))=\sum_{k=1}^{2r}{N_{2k}(\Gamma(p,q))}.$$  Again we apply Lemma \ref{lemmaNWt} to get 
\begin{eqnarray*}
\sum_{k=1}^{2r}{\frac{q-(2k)}{q-1}\cdot\frac{p}{q}}-2r & \leq & N_{even} \leq \sum_{k=1}^{2r}{\frac{q-(2k)}{q-1}\cdot\frac{p}{q}}+2r.
\end{eqnarray*}
Rearranging and simplifying yields
\begin{eqnarray*}
\frac{p}{q-1}\cdot\frac{q-2}{4}-\frac{q}{2} & \leq & N_{even} \leq \frac{p}{q-1}\cdot\frac{q-2}{4}+\frac{q}{2}.
\end{eqnarray*}
Next apply Lemma \ref{lemmaN} to get
\begin{eqnarray*}
\frac{\frac{p}{q-1}\cdot\frac{q-2}{4}-\frac{q}{2}}{\frac{p}{2}+\frac{q}{2}} & \leq & \frac{N_{even}}{N} \leq \frac{\frac{p}{q-1}\cdot\frac{q-2}{4}+\frac{q}{2}}{\frac{p}{2}-\frac{q}{2}}\\
\frac{p(q-2) - 2q(q-1)}{2(p+q)(q-1)} &\leq& \frac{N_{even}}{N} \leq \frac{p(q-2) + 2q(q-1)}{2(p-q)(q-1)}.
\end{eqnarray*}
Take limit as $p \to \infty$
\begin{eqnarray*}
\frac{q-2}{2(q-1)} &\leq& \lim_{p\to \infty}{\frac{N_{even}}{N}} \leq \frac{q-2}{2(q-1)}.
\end{eqnarray*}
Thus 
\begin{eqnarray*}
\lim_{p\to \infty}{\frac{N_{odd}}{N}} = \frac{q-2}{2(q-1)}.
\end{eqnarray*} 
While we have shown only one case, the others are similar; in Table \ref{EvenOddTable} we summarize the other cases.

\begin{table}[ht]
\caption{Summary of other cases.}
{\begin{tabular}{|c|c|c|}
\hline
$q$ & $\lim_{p}{\frac{N_{even}}{N}}$ & $\lim_{p}{\frac{N_{odd}}{N}}$\\
\hline
$0 \pmod{4}$ & $\frac{q-2}{2(q-1)}$& $\frac{q}{2(q-1)}$\\
$1 \pmod{4}$ & $\frac{q-1}{2q}$& $\frac{q+1}{2q}$\\
$2 \pmod{4}$ & $\frac{q-2}{2(q-1)}$& $\frac{q}{2(q-1)}$\\
$3 \pmod{4}$ & $\frac{q-1}{2q}$& $\frac{q+1}{2q}$\\
\hline
\end{tabular}}
\label{EvenOddTable}
\end{table}

Taking the limit as $q$ goes to infinity in all cases gives the desired result $$\lim_{q\to \infty}{\lim_{p\to \infty}{\frac{N_{odd}}{N}}} = \frac{1}{2}.$$
Hence the limit the ratio of the number of even weight terms to the total number of terms equals the limit of the ratio of the number of odd weight terms to the total number of terms.
\end{proof}

So far we have ignored the sign of coefficients.  Using our notion of weight, we restate a theorem from \cite{LWW}.
\begin{theorem}{(Loehr, Warrington, Wilf  \cite{LWW})} \label{LWW} The coefficient $c_{r,s}$ of the weight $w$ monomial $x^r y^s$ in $f_{p,q}$ is positive when $\gcd \left( r,s,w \right)$ is odd, the coefficient is negative when $\gcd \left( r,s,w \right)$ is even.
\end{theorem}

\begin{corollary}  \label{Cor1} The odd weight terms in $f_{p,q}$ are all positive, and the even weight terms in $f_{p,q}$ alternate sign.
\end{corollary}
\begin{proof}  For odd weight terms $\gcd \left( r,s,w \right)$ is always odd, thus by Theorem \ref{LWW}, the coefficients are all positive.

When $w$ is even, $\gcd \left( r,s,w \right)$ is even whenever both $r$ and $s$ are even.  But $r=wp-qs$, so $r$ is even if $s$ and $w$ are even.  Finally for fixed weight $w$ the possible $s$-values are consecutive integers.  Hence for $w$ even, $s$ will alternate between even and odd values, thereby making $\gcd \left( r,s,w \right)$ alternate between even and odd values.  By Theorem \ref{LWW} the terms of odd weight in $f_{p,q}$ will alternate signs.
\end{proof}

For ease of notation we make the following definition.

\begin{definition}  Let $T(q)$ denote the asymptotic positivity ratio for $\Gamma(p,q)$, then $$T(q)=\lim_{p\to \infty}{L(\Gamma(p,q))}.$$
\end{definition}

The sequence $T(q)$ is of interest.  We list the first few terms:

$$T(q)=\left( 1, \; 1, \; \frac{5}{6}, \; \frac{5}{6}, \; \frac{4}{5}, \; \frac{4}{5}, \; \frac{11}{14}, \; \frac{11}{14}, \; \frac{7}{9}, \cdots \right).$$

In Corollary \ref{c:apvp} we show that the sequence $T(q)$ is monotone non-increasing, and each value repeats twice.

Now we combine Theorem \ref{LWW} with our previous estimates on the number of terms of each weight to compute the asymptotic positivity ratio.
\begin{proposition}  \label{lemmaSq}  The limit in the definition of the asymptotic positivity ratio exists, and
\begin{equation}\label{e:rational}T(q)=\begin{cases}
\frac{3q+1}{4q} & \text{if $q$ is odd},\\
\frac{3q-2}{4(q-1)} & \text{if $q$ is even}.
\end{cases}\end{equation}
\end{proposition}
\begin{proof}  First consider the even and odd weights separately.  
\begin{eqnarray}
T(q)=\lim_{p\to \infty}{\frac{N^{+}(\Gamma(p,q))}{N(\Gamma(p,q))}} & = & \lim_{p\to \infty}{\frac{N_{odd}^{+}(\Gamma(p,q))+N_{even}^{+}(\Gamma(p,q))}{N(\Gamma(p,q))}}\\
& = &  \lim_{p\to \infty}{\frac{N_{odd}^{+}(\Gamma(p,q))}{N(\Gamma(p,q))}}+\lim_{p\to \infty}{\frac{N_{even}^{+}(\Gamma(p,q))}{N(\Gamma(p,q))}}. \label{eq19}
\end{eqnarray}
By Corollary \ref{Cor1} all odd weight terms are positive and the even weight terms alternate in sign.  Hence $$N_{odd}=N_{odd}^{+},$$ and  
\begin{eqnarray}{\lim_{p\to \infty}{\frac{N_{even}^{+}(\Gamma(p,q))}{N(\Gamma(p,q))}}}=\frac{1}{2}\cdot{\lim_{p\to \infty}{\frac{N_{even}(\Gamma(p,q))}{N(\Gamma(p,q))}}}. \label{eq20}
\end{eqnarray}

When $q$ is even, combining equations \ref{eq19} and \ref{eq20} with table \ref{EvenOddTable} in Lemma \ref{lemma2} yields $$T(q)=\lim_{p\to \infty}{L(\Gamma(p,q))}=\frac{q}{2(q-1)}+\frac{1}{2}\cdot\frac{q-2}{2(q-1)}=\frac{3q-2}{4(q-1)}.$$

Similarly when $q$ is odd, we have $$T(q)=\lim_{p\to \infty}{L(\Gamma(p,q))}=\frac{q+1}{2q}+\frac{1}{2}\cdot\frac{q-1}{2q}=\frac{3q+1}{4q}.$$
\end{proof}

\begin{corollary} \label{c:apvp}
\mbox{}\\
\noindent (i)  The sequence $T(q)$ is monotone.\\
\noindent (ii)  The limit as $q$ goes to infinity of $T(q)$ exists.\\
\noindent (iii)  $T(2r-1)=T(2r)$ for all $r\in \mathbb{Z}^{+}$.\\
\end{corollary}

Now taking the limit as $q$ goes to infinity in Proposition \ref{lemmaSq} gives one of our main results.

\begin{theorem}  Let $T(q)$ denote the asymptotic positivity ratio of $\Gamma(p,q)$, then
$$\lim_{q\to \infty}{T(q)}=\frac{3}{4}.$$
\end{theorem}

\section{Dihedral Group}
The analysis can be extended to other groups in $U(2)$; for example, in this section, we define families of unitary representations of dihedral groups, and we determine the asymptotic positivity ratios to be $\frac{1}{2}$.

Let $D_p$ denote the dihedral group with $2p$ elements; namely, $$D_p:=\left< a, b \; | \; a^p=b^2=1, bab=a^{-1} \right>.$$  Without loss of generality, let $\iota : D_p \to U(2)$ be the faithful representation generated by 
\begin{eqnarray*}
\iota(a) & = & 
\begin{pmatrix} 
\omega & 0\\ 
0 & \omega^{-1}
\end{pmatrix}\\
\iota(b) & = & 
\begin{pmatrix} 
0 & 1\\ 
1 & 0
\end{pmatrix}.\\
\end{eqnarray*}  Here the $a$ corresponds to rotation, and the $b$ corresponds to reflection.  Let $$\Delta_p = \iota(D_p).$$

Before stating the main results of this section we begin with an example.  We compute the number of positive and negative eigenvalues of $\Phi_{\Delta_3}(z, \bar{z})$.  Expanding the product in the definition we get 
\begin{eqnarray*}
\lefteqn{\Phi_{\Delta_3} =}\\
& & z_1^3 \bar{z_1}^3+z_2^3\bar{z_1}^3-z_1^3 z_2^3 \bar{z_1}^6+6z_1 z_2 \bar{z_1} \bar{z_2} - 3z_1^4 z_2 \bar{z_1}^4 \bar{z_2} - 3z_1 z_2^4 \bar{z_1}^4 \bar{z_2}-9z_1^2 z_2^2 \bar{z_1}^2 \bar{z_2}^2 \\
& &+ z_1^3 \bar{z_2}^3+z_2^3\bar{z_2}^3-z_1^6\bar{z_1}^3 \bar{z_2}^3-z_2^6 \bar{z_1}^3 \bar{z_2}^3 - 3 z_1^4 z_2 \bar{z_1} \bar{z_2}^4 -3 z_1 z_2^4 \bar{z_1} \bar{z_2}^4 -z_1^3 z_2^3 \bar{z_2}^6.\\
\end{eqnarray*}
In contrast to the cyclic case we get off-diagonal terms, and hence it is not enough to simply count the number of terms to get the number of eigenvalues.  Rewriting in terms of a polynomials invariant under the $D_3$-action, we get
\begin{eqnarray*}
\lefteqn{\Phi_{\Delta_3} =}\\
& & (z_1^3+z_2^3)(\bar{z_1}^3 +\bar{z_2}^3)- z_1^3 z_2^3 (\bar{z_1}^6+\bar{z_2}^6)- \bar{z_1}^3 \bar{z_2}^3 (z_1^6 +z_2^6)+ 6z_1 z_2 \bar{z_1} \bar{z_2}\\ 
& &-9z_1^2 z_2^2 \bar{z_1}^2 \bar{z_2}^2-3 z_1 z_2 (z_1^3 +z_2^3) \bar{z_1} \bar{z_2} (\bar{z_1}^3 +\bar{z_2}^3).\\
\end{eqnarray*}

Equivalently we get $$\Phi_{\Delta_3} =
\begin{pmatrix}
\bar{z_1}^3 + \bar{z_2}^3\\
\bar{z_1} \bar{z_2} (\bar{z_1}^3 +\bar{z_2}^3)\\
\bar{z_1} \bar{z_2}\\
\bar{z_1}^2 \bar{z_2}^2\\
\bar{z_1}^3 \bar{z_2}^3\\
\bar{z_1}^6+\bar{z_2}^6\\
\end{pmatrix}^{T}
\begin{pmatrix}
1 &  0 & 0 &  0 & 0 & 0\\
0 & -3 & 0 &  0 & 0 & 0\\
0 &  0 & 6 &  0 & 0 & 0\\
0 &  0 & 0 & -9 & 0 & 0\\
0 &  0 & 0 &  0 & 0 & -1\\
0 &  0 & 0 &  0 & -1 & 0\\
\end{pmatrix}
\begin{pmatrix}
z_1^3 + z_2^3\\
z_1 z_2 (z_1^3 +z_2^3)\\
z_1 z_2\\
z_1^2 z_2^2\\
z_1^3 z_2^3\\
z_1^6+z_2^6\\
\end{pmatrix}.$$

Hence the eigenvalues of $\Phi_{\Delta_3}$ are 1, -3, 6, -9, 1, -1.  Then $$S(\Delta_3)=(3,3).$$

We proceed along these lines for general $p$. We compute the asymptotic positivity ratio in the following theorem.

\begin{theorem}  Let $\Delta_p$ be a dihedral group of order $2p$ in $U(2)$, then
$$\lim_{p\to \infty}{L(\Delta_p)}=\frac{1}{2}.$$
\end{theorem}
\begin{proof}  In order to prove this theorem we first invoke Theorem \ref{Dihedral} relating $\Phi_{\Delta_p}$ to the more familiar $f_{p,p-1}$.  In Lemma \ref{l:DNumTerms} below we count the number of positive and negative eigenvalues.  Below in Corollary \ref{c:DRatio}, we compute the positivity ratio, and we show that the limit as $p$ goes to infinity exists.  The conclusion of this theorem follows by taking the limit as $p$ goes to infinity in Corollary \ref{c:DRatio}.
\end{proof}

D'Angelo \cite{D4} proves the following result relating $\Phi_{\Delta_p}$ to $f_{p,p-1}$.  The key idea is that the elements of $\Delta_p$ are either diagonal matrices or anti-diagonal matrices, and hence we can consider them separately as $f_{p,p-1}$ evaluated at different points.
\begin{theorem}{({D'Angelo} \cite{D4})}  \label{Dihedral} The invariant polynomial corresponding to the representation $\iota$ satisfies: $$\Phi_{\Delta_p} =  f_{p,p-1}(\left|z_1 \right|^2, \left|z_2 \right|^2)+
f_{p,p-1}(z_2 \bar{z_1}, z_1 \bar{z_2})-f_{p,p-1}(\left|z_1 \right|^2,\left|z_2 \right|^2)f_{p,p-1}(z_2 \bar{z_1}, z_1 \bar{z_2}).$$
\end{theorem}

Unlike the general cyclic case, here we can exactly determine the numbers of positive and negative eigenvalues.

\begin{lemma}  \label{l:DNumTerms}  The total number of eigenvalues is
$$N(\Delta_{p})=p+\left \lfloor \frac{p}{2} \right \rfloor + 2.$$ 
The number of positive eigenvalues is $$N^{+}(\Delta_{p})=\left \lfloor \frac{p}{2} \right \rfloor +\left \lfloor \frac{p}{4} \right \rfloor +2.$$
\end{lemma}
\begin{proof}  Recall \begin{equation} \label{e:fpminus1}f_{p,p-1}(x,y)=x^p + y^p + \sum_{j=1}^{\left \lfloor \frac{p}{2} \right \rfloor}{(-1)^j c_{p,j} x^j y^j}.\end{equation}  Let $$B_p(x,y)=\sum_{j=1}^{\left \lfloor \frac{p}{2} \right \rfloor}{(-1)^j c_{p,j} x^j y^j}.$$

We invoke Theorem \ref{Dihedral} to decompose $\Phi_{\Delta_p}$; namely,
\begin{equation} \label{e:decomp}
\begin{split}
\Phi_{\Delta_p}(z,\bar{z}) &=  (z_1 \bar{z_1})^p +(z_2 \bar{z_2})^p + B_p(z_1 \bar{z_1},z_2 \bar{z_2})+(z_2 \bar{z_1})^p +(z_1 \bar{z_2})^p + B_p(z_2 \bar{z_1},z_1 \bar{z_2}) \\
& -(z_1 \bar{z_1})^p (f_{p,p-1}(z_2 \bar{z_1}, z_1 \bar{z_2}))-(z_2 \bar{z_2})^p(f_{p,p-1}(z_2 \bar{z_1}, z_1 \bar{z_2}))\\
&- (z_2 \bar{z_1})^pB_p(z_1 \bar{z_1}, z_2 \bar{z_2}) -(z_1 \bar{z_2})^p B_p(z_1 \bar{z_1}, z_2 \bar{z_2})\\
&-  B_p(z_1 \bar{z_1}, z_2 \bar{z_2})B_p(z_2 \bar{z_1}, z_1 \bar{z_2}).
\end{split}
\end{equation}

First notice $$B_p(z_1 \bar{z_1},z_2 \bar{z_2})+B_p(z_2 \bar{z_1},z_1 \bar{z_2})=2 B_p(z_1 \bar{z_1},z_2 \bar{z_2}).$$

Second we expand the last term in \eqref{e:decomp} to get
$$B_p(z_1 \bar{z_1}, z_2 \bar{z_2})B_p(z_2 \bar{z_1}, z_1 \bar{z_2})= \sum_{k=2}^{2 \left \lfloor \frac{p}{2} \right \rfloor}{(-1)^k \sum_{\substack{a+b=k\\ 1\leq a,b \leq \left \lfloor \frac{p}{2} \right \rfloor}}{c_{p,a}c_{p,b}}(z_1 \bar{z_1} z_2 \bar{z_2})^k}.$$

Define $$E_k=
\sum_{\substack{a+b=k\\ 1\leq a,b \leq \left \lfloor \frac{p}{2} \right \rfloor}}{c_{p,a}c_{p,b}}+2c_{p,k}$$ for $1 \leq k \leq \left \lfloor \frac{p}{2} \right \rfloor$, and $$E_k=
\sum_{\substack{a+b=k\\ 1\leq a,b \leq \left \lfloor \frac{p}{2} \right \rfloor}}{c_{p,a}c_{p,b}}$$ for $\left \lfloor \frac{p}{2} \right \rfloor < k \leq 2 \left \lfloor \frac{p}{2} \right \rfloor$.

Observe that $$E_k>0$$ for all $1 \leq k \leq 2 \left \lfloor \frac{p}{2} \right \rfloor$.

Now we want to write $\Phi_{\Delta_p}$ in terms of invariant polynomials.  The polynomials $$z_1^p+z_2^p,\;  z_1^j z_2^j,\; z_1^j z_2^j(z_1^p+z_2^p), \; z_1^{2p}+z_2^{2p}$$ are linearly independent and invariant under the $D_p$-action.  Writing $\Phi_{\Delta_p}$ in terms of these invariant polynomials we get
\begin{eqnarray*}
\Phi_{\Delta_p}(z,\bar{z}) &=& (z_1^p+z_2^p)(\bar{z_1}^p+\bar{z_2}^p) + \sum_{k=1}^{2 \left \lfloor \frac{p}{2} \right \rfloor}{(-1)^{k+1}E_k(z_1 z_2 \bar{z_1} \bar{z_2})^k}-z_1^p z_2^p(\bar{z_1}^{2p}+\bar{z_2}^{2p})\\
& &  - \bar{z_1}^p \bar{z_2}^p (z_1^{2p}+z_2^{2p})+\sum_{j=1}^{\left \lfloor \frac{p}{2} \right \rfloor}{(-1)^{j} c_{p,j} z_1^j z_2^j (z_1^p +z_2^p)\bar{z_1}^j \bar{z_2}^j(\bar{z_1}^p+\bar{z_2}^p)}.\\
\end{eqnarray*}

The underlying Hermitian matrix is nearly diagonal; we have only 2 non-diagonal terms.  We now explicitly write out the polynomial in matrix form.  Let
$$b=\begin{pmatrix}
z_1^{p}+z_2^{p}\\
z_1^j z_2^j(z_1^{p}+z_2^{p})\\
z_1^{k} z_2^{k}\\
z_1^{2p}+z_2^{2p}\\
\end{pmatrix}$$  for $1\leq j \leq \left \lfloor \frac{p}{2} \right \rfloor$ and $1\leq k \leq p$. Also let $$H_p=\begin{pmatrix}
1 & 0   & 0 & 0\\
0 & A_{p,1} & 0 & 0\\
0 & 0 & A_{p,2} & 0\\
0 & 0 & 0   & A_{p,3}\\
\end{pmatrix}.$$   The $\left \lfloor \frac{p}{2} \right \rfloor$ by $\left \lfloor \frac{p}{2} \right \rfloor$ diagonal matrix $A_{p,1}$ has diagonal entries given by $$(A_{p,1})_{jj}=(-1)^{j} c_{p,j}.$$  Next the matrix $A_{p,2}$ is $p$ by $p$ diagonal with diagonal entries given by $$(A_{p,2})_{kk}=(-1)^{k+1}E_k.$$  Finally the matrix 
$$A_{p,3}=\begin{pmatrix}
0 & -1\\
-1& 0\\
\end{pmatrix}$$ when $p$ is odd, and
$$A_{p,3}=\begin{pmatrix}
(-1)E_p& -1\\
-1& 0\\
\end{pmatrix}$$ when $p$ is even.

Thus $$\Phi_{\Delta_p}(z,\bar{z})=b^{*}H_pb.$$

Now we just count the eigenvalues of each diagonal submatrix.  The matrix $A_{p,1}$ has $\left \lfloor \frac{p}{2} \right \rfloor$ eigenvalues and $\left \lfloor \frac{p}{4} \right \rfloor$ positive eigenvalues.  The matrix $A_{p,2}$ has $p$ eigenvalues and $\left \lfloor \frac{p}{2} \right \rfloor$ positive eigenvalues.  In either case the matrix $A_{p,3}$ has 1 positive eigenvalue and 1 negative eigenvalue.  Finally adding up the eigenvalues for each submatrix we get the desired result.
\end{proof}

\begin{corollary}  \label{c:DRatio} The following positivity ratios hold for the dihedral group: $$L(\Delta_p)=\begin{cases}
\frac{1}{2}+\frac{2}{3p+4} & p \equiv 0 \bmod 4,\\
\frac{1}{2}+\frac{1}{3p+3} & p \equiv 1 \bmod 4,\\
\frac{1}{2}+\frac{1}{3p+4} & p \equiv 2 \bmod 4,\\
\frac{1}{2} & p \equiv 3 \bmod 4.
\end{cases}$$ Moreover, the limit as $p$ goes to infinity of the ratio equals $\frac{1}{2}$.
\end{corollary}
\begin{proof}  The four cases are similar.  We consider only the case where $p\equiv 0 \bmod 4$.  By Lemma \ref{l:DNumTerms} it follows that 
\begin{equation*}
N(\Delta_p)=p +\frac{p}{2} +2=\frac{3p+4}{2}
\end{equation*}
and
\begin{equation*}
N^{+}(\Delta_p)=\frac{p}{2} +\frac{p}{4}+2=\frac{3p+8}{4}.
\end{equation*}
Hence the ratio is:
$$L(\Delta_p)=\frac{N^{+}(\Delta_p)}{N(\Delta_p)}=\frac{3p+8}{2(3p+4)}=\frac{1}{2}+\frac{2}{3p+4}.$$
\end{proof}

\section{Orbit Polynomial and Chern Orbit Classes}
The goal of this section is to describe how the group-invariant Hermitian polynomials $\Phi_\Gamma$ arise in the context of representation theory.  Theorem \ref{thm:chern} will express the invariant polynomial $\Phi_\Gamma$ in terms of the orbit Chern classes.  We begin by recalling some basic definitions.  In particular, we define the orbit polynomial and orbit Chern classes as in \cite{S0}.  

Let $\pi : G \to U(n)$ be a unitary representation of a finite group $G$.  Let $\mathbb{C}[z_1,\cdots, z_n]$ denote the polynomial algebra in $n$ variables over $\mathbb{C}$.  We define a group action on the polynomial algebra by \[(g \cdot h)(z_1,\cdots,z_n) = h(\pi(g^{-1})(z_1,\cdots,z_n))\] where $h\in \mathbb{C}[z_1,\cdots, z_n]$ and $g \in G$.  The set of fixed points of this action is the set of group-invariant polynomials in $\mathbb{C}[z_1,\cdots,z_n]$.  Denote the set of fixed points of the action by $$\mathbb{C}[z_1,\cdots,z_n]^{G} = \left\{ h\in \mathbb{C}[z_1,\cdots,z_n] : g \cdot h = h \; \forall g \in G \right\}.$$  Despite this notation the set of fixed points depends on the representation of the group.

Define $G \cdot h$ to be the $G$-orbit corresponding to $h \in \mathbb{C}[z_1, \cdots,z_n]$.  Following \cite{S0} define the orbit polynomial of $G \cdot h$ by $$\phi_{G\cdot h}(X)= \prod_{b \in G \cdot h}{\left( X+b \right)}.$$  Expanding the product we get $$\phi_{G \cdot h}(X)=\sum_{a+b=\left| G \right|}{c_a(G \cdot h)X^b}$$ where $c_a(G \cdot h) \in \mathbb{C}[z_1,\cdots,z_n]^{G}$ are called the orbit Chern classes of the orbit $G \cdot h$.  The definition of orbit Chern class agrees with the usual topological definition of Chern class; this construction is given in \cite{SS}.

Now we restrict our attention to the case $n=2$, $G$ is a cyclic group of order $p$, and the representation is given by $\pi(G)= \Gamma(p,q)$.  Let $h=-(z_1+z_2)$, then $G \cdot h = \left\{ -\omega^j z_1 - \omega^{q j} z_2: j = 0, \cdots p-1 \right\}$.  The orbit polynomial $\phi_{G\cdot h}(X)= \prod_{j=0}^{p-1}{\left( X-\omega^j z_1 - \omega^{q j} z_2 \right)}.$  Thus the orbit polynomial evaluated at 1 is $f_{p,q}$.  Taking the total Chern class of the orbit $G\cdot h$ is exactly $f_{p,q}$.

We can polarize $\Phi_{\Gamma}$; we treat $z$ and $\bar{z}$ as independent variables.  Thus we write $$\Phi_{\Gamma}(z, \bar{w})=1-\prod_{\gamma \in \Gamma}{\left(1-\langle \gamma z, w \rangle \right)}.$$  If we set $w= (1,1,\cdots,1)$, then $\Phi_{\Gamma}(z, \bar{w})=1-\prod_{\gamma \in \Gamma}{\left(1-\sum_{j=1}^{n}{\gamma z_j \bar{w_j}} \right)},$  which is exactly the alternating sum of orbit Chern classes of the orbit corresponding to $z_1+z_2+\cdots+z_n$.  Therefore we have the following theorem.

\begin{theorem}  \label{thm:chern}  Let $\pi : G \to U(n)$ be a faithful, unitary representation of the finite group $G$.  Put $\Gamma=\pi(G)$.  Then $$\Phi_{\Gamma}(z, \bar{z})=\sum_{j=1}^{p}{(-1)^{j-1} c_j(G \cdot (z_1+\cdots +z_n))}.$$
\end{theorem}

In particular we have the following corollary.

\begin{corollary}  \label{c:last} Suppose $\pi : G \to \Gamma(p,q)$ is the representation given above, then $$f_{p,q}(x,y)=\sum_{j=1}^{p}{(-1)^{j-1} c_j(G \cdot (x+y))}.$$
\end{corollary}

\newpage

\appendix
\section{Source Code for Computing Signature Pairs in Mathematica}

Let $\Gamma$ be a finite subgroup of $U(2)$, and recall \begin{equation*}\Phi_{\Gamma}(z, \bar{w})=1-\prod_{\gamma \in \Gamma}{\left(1-\langle \gamma z, w \rangle \right)}.\end{equation*}  We introduce the Mathematica \cite{Mathematica} function GroupSignaturePair.  This function uses standard Mathematica commands to compute the eigenvalues of $\Phi_{\Gamma}$.  This function takes a list of the group elements of $\Gamma$ and returns a list of the eigenvalues of the underlying hermitian matrix of the polynomial $\Phi_{\Gamma}$.  Computing signature pairs in this way is very memory intensive.  To improve performance, one can use the Mathematica command $N[]$ to numerically find the eigenvalues.

\begin{verbatim}
GroupSignaturePair[group_] := 
Module[{matrix, hermitianmatrix, poly, order, eigenvalues},
  poly = First[
Expand[1 - Product[1 - 
Transpose[(L.{{z1}, {z2}})].{w1, w2}, {L, group}]]];
  order = Length[group];
  matrix = CoefficientList[poly, {z1, z2, w1, w2}];
  hermitianmatrix = Partition[Flatten[matrix], (order + 1)^2];
  eigenvalues = Eigenvalues[hermitianmatrix];
  Return[eigenvalues]]
\end{verbatim}

\section*{Acknowledgements}
I acknowledge support from NSF grant DMS 08-38434 ``EMSW21-MCTP: Research Experience for Graduate Students" and from NSF grant DMS 07-53978 of John D'Angelo.  I also want to thank him for sharing this interesting problem with me, for many helpful discussions about these ideas, and for providing the statement and proof of Lemma 3.3.  I would also like to thank the organizers of the Park City Math Institute for giving me the opportunity to give a short talk on this topic at the 2008 summer session on algebraic and analytic geometry.  The results of this paper and similar results for other groups will appear in my doctoral thesis at the University of Illinois, Urbana-Champaign.  I would also like to thank the referee for carefully reading this paper and for giving many useful comments.

\bibliographystyle{plain}
\bibliography{References}
\end{document}